\documentclass[12pt,reqno]{article}

\usepackage[usenames]{color}
\usepackage{amssymb}
\usepackage{graphicx}
\usepackage{amscd}

\usepackage[colorlinks=true,
linkcolor=webgreen,
filecolor=webbrown,
citecolor=webgreen]{hyperref}

\definecolor{webgreen}{rgb}{0,.5,0}
\definecolor{webbrown}{rgb}{.6,0,0}

\usepackage{color}
\usepackage{fullpage}
\usepackage{float}

\usepackage{graphics,amsmath,amssymb}
\usepackage{amsthm}
\usepackage{amsfonts}
\usepackage{latexsym}
\usepackage{epsf}

\newcommand{\seqnum}[1]{\href{http://oeis.org/#1}{\underline{#1}}}

\begin{document}

\theoremstyle{plain}
\newtheorem{theorem}{Theorem}
\newtheorem{corollary}[theorem]{Corollary}
\newtheorem{lemma}[theorem]{Lemma}
\newtheorem{proposition}[theorem]{Proposition}

\theoremstyle{definition}
\newtheorem{definition}[theorem]{Definition}
\newtheorem{example}[theorem]{Example}
\newtheorem{conjecture}[theorem]{Conjecture}

\theoremstyle{remark}
\newtheorem{remark}[theorem]{Remark}

\begin{center}
\vskip 1cm{\LARGE\bf A Variation on Mills-Like \\
\vskip .1in Prime-Representing Functions
}
\vskip 1cm
\large
L\'aszl\'o T\'oth\\
Rue des Tanneurs 7 \\
L-6790 Grevenmacher \\
Grand Duchy of Luxembourg \\
\href{mailto:uk.laszlo.toth@gmail.com}{\tt uk.laszlo.toth@gmail.com}
\end{center}

\vskip .2 in

\begin{abstract}
Mills showed that there exists a constant $A$ such that
$\lfloor{A^{3^n}}\rfloor$ is prime for every positive integer $n$.
Kuipers and Ansari generalized this result to $\lfloor{A^{c^n}}\rfloor$
where $c\in\mathbb{R}$ and $c\geq 2.106$. The main contribution of this paper
is a proof that the function $\lceil{B^{c^n}}\rceil$ is also a
prime-representing function, where $\lceil X\rceil$ denotes the ceiling
or least integer function.
Moreover, the first 10 primes in
the sequence generated in the case $c=3$ are calculated. Lastly, the
value of $B$ is approximated to the first $5500$ digits and is shown to
begin with $1.2405547052\ldots$.
\end{abstract}

\section{Introduction}
Mills \cite{Mills47} showed in 1947 that there exists a constant $A$ 
such that $\lfloor{A^{3^n}}\rfloor$ is prime for all positive integers $n$. Kuipers \cite{Kuipers50} and Ansari \cite{Ansari51} generalized this result to all $\lfloor{A^{c^n}}\rfloor$ where $c\in\mathbb{R}, c\geq2.106$, i.e., there exist infinitely many $A$'s such 
that the above expression yields a prime for all positive integers $n$. Caldwell and Cheng \cite{CaldwellCheng05} calculated the minimum constant $A$ for the case $c=3$ up to the first $6850$ digits (\seqnum{A051021}), and found it to be approximately equal to $1.3063778838\ldots$. This process involved computing the first $10$ primes $b_i$ in the sequence generated by the function (\seqnum{A051254}), with $b_{10}$ having 6854 decimal digits.

The main contribution of this paper is a proof that the function $\lceil{B^{c^n}}\rceil$ satisfies the same criteria, where $\lceil X\rceil$ denotes the 
ceiling function (the least integer greater than or equal to $X$).
In other words, there exists a constant $B$ 
such that for all positive integers $n$, the expression
$\lceil{B^{c^n}}\rceil$ yields a prime for $c\geq 3, c\in\mathbb{N}$. Moreover, the sequence of primes generated by such functions is monotonically increasing. Lastly, analogously to \cite{CaldwellCheng05} the case $c=3$ is studied in more detail and the value of $B$ is approximated up to the first $5500$ decimal digits by calculating the first $10$ primes $b_i$ of the sequence. 

In contrast to Mills' formula and given that here the floor function is
replaced by a ceiling function, the process of generating the prime
number sequence $P_0, P_1, P_2, \ldots$ involves taking the greatest
prime smaller than $P_n^c$ at each step instead of smallest prime
greater than $P_n^c$,  in order to find $P_{n+1}$. As a consequence,
the sequence of primes generated by $\lceil{B^{c^n}}\rceil$ is
different from the one generated by $\lfloor{A^{c^n}}\rfloor$ for the
same value of $c$ and the same starting prime (apart from the first
element of course).

\section{The prime-representing function}

This paper begins with a proof of the case $c=3$ and will proceed to a generalization of the function to all $c\geq 3, c\in\mathbb{N}$.

By using Ingham's result \cite{Ingham37} on the difference of consecutive primes:
$$
p_{n+1} - p_n < Kp_n^{5/8},
$$
and analogously to Mills' reasoning \cite{Mills47}, we construct an infinite sequence of primes $P_0, P_1, P_2, \ldots$ such that $\forall n \in \mathbb{N} : (P_n-1)^3+1 < P_{n+1} < P_n^3$ using the following lemma.

\begin{lemma}\label{bounds}
$\forall N > K^8+1 \in \mathbb{N} : \exists p \in \mathbb{P} : (N-1)^3+1<p<N^3$, where $\mathbb{P}$ denotes the set of prime numbers.
\end{lemma}
 
\begin{proof}
Let $p_{n}$ be the greatest prime smaller than $(N-1)^3$.

\begin{align*}
(N-1)^3 & < p_{n+1}\\
   	    & < p_n + Kp_n^{5/8} \\
        & < (N-1)^3 + K\left((N-1)^3\right)^{5/8} && (\text{since} \ p_n < (N-1)^3) \\
        & < (N-1)^3 + (N-1)^2 					  && (\text{since} \ N > K^8 +1) \\
        & < N^3 - 2N^2 + N \\
        & < N^3.
\end{align*}

Note that since $(N-1)^3 < p_{n+1}$, $(N-1)^3+1 < p_{n+1}$ since $(N-1)^3+1 = N(N^2-3N+3)$ is not prime.

\end{proof}

Given the above we can construct an infinite sequence of primes $P_0, P_1, P_2, \ldots$ such that for every positive integer $n$, we have: $(P_n-1)^3+1 < P_{n+1} < P_n^3$.

We now define the following two functions:

\begin{align*}
\forall n \in \mathbb{Z^+}: u_n &= (P_n-1)^{3^{-n}}, \\
\forall n \in \mathbb{Z^+}: v_n &= P_n^{3^{-n}}.
\end{align*}

The following statements can immediately be deduced:

\begin{itemize}
\item $u_n < v_n$,
\item $u_{n+1} = (P_{n+1}-1)^{3^{-n-1}} > \left((P_n-1)^3+1)-1\right)^{3^{-n-1}} = (P_n-1)^{3{-n}} = u_n$,
\item $v_{n+1} = P_{n+1}^{3^{-n-1}} < (P_n^3)^{3^{-n-1}} = P_n^{3^{-n}} = v_n$.
\end{itemize} 

It follows that $u_n$ forms a bounded and monotone increasing sequence.

\begin{theorem} \label{theorem3n}
There exists a positive real constant $B$ such that $\lceil{B^{3^n}}\rceil$ is a prime-representing function for all positive integers $n$.
\end{theorem}

\begin{proof}
Since $u_n$ is bounded and strictly monotone, there exists a number $B$ such that
$$
B := \lim_{n\rightarrow\infty}u_n.
$$

From the above deduced properties of $u_n$ and $v_n$, we have

\begin{alignat*}{2}
u_n &< B &&< v_n, \\
(P_n-1)^{3^{-n}} &< B &&< P_n^{3^{-n}}, \\
P_{n}-1 &< B^{3^n} &&< P_n.
\end{alignat*}

\end{proof}

\begin{theorem}
There exists a positive real constant $B$ such that $\lceil{B^{c^n}}\rceil$ is a prime-representing function for $c\geq 3, c\in\mathbb{N}$ and all positive integers $n$.
\end{theorem}

\begin{proof}

We can use the generalizations to Mills' function as shown by Kuipers \cite{Kuipers50} and Dudley \cite{Dudley69} in order to show that $\lceil{B^{c^n}}\rceil$ is also a prime-representing function for $c\geq 3, c\in\mathbb{N}$. This proof is short as it is essentially identical to the one presented above, with the following modifications.

As shown by Kuipers \cite{Kuipers50} for Mills' function, we first define $a=3c-4, b=3c-1$. Therefore $a/b\geq 5/8$. This means that in Ingham's equation there exists a constant $K'$ such that
$$
p_{n+1} - p_n < K'p_n^{a/b}.
$$

Lemma \ref{bounds} can then be modified by taking $N>K'^b+1$, defining $p_n$ as the greatest prime smaller than $(N-1)^c$ and noticing that $ca+1 = b(c-1)$. Analogously to the proof in Lemma \ref{bounds}, we quickly obtain the bounds $(N-1)^c+1<p<N^c$. This means that we can construct a sequence of primes $P_0, P_1, P_2, \ldots$ such that for every positive integer $n$, $(P_n-1)^c+1 < P_{n+1} < P_n^c$.

This is then concluded with a similar reasoning as in the proof of Theorem \ref{theorem3n}.

\end{proof}

\section{Numerical calculation of $B$}

In this section, a numerical approximation of $B$ is presented for the case $c=3$. Mills \cite{Mills47} suggested using the lower bound $K=8$ for the first prime in the classical Mills function $\lfloor{A^{3^n}}\rfloor$, where $K$ is the constant defined in Ingham's paper \cite{Ingham37}. Other authors, including Caldwell and Cheng \cite{CaldwellCheng05}, decided to begin with the prime $2$ and then choose the least possible prime at each step. In this case, since the ceiling function replaces the floor function, we choose the greatest possible prime smaller than $P_n^3$ as the next element $P_{n+1}$.

If $p_i$ denotes the $i^{\rm th}$ prime in the sequence, we obtain

\begin{itemize}
\item $p_1 = 2$
\item $p_2 = 7$
\item $p_3 = 337$
\item $p_4 = 38272739$
\item $p_5 = 56062005704198360319209$
\item $p_6 = 17619999581432728735667120910458586439705503907211069\backslash\\
6028654438846269$
\item $p_7 = 54703823381492990628407924713718713957740513297193414\backslash\\
21259587335767096542227048457036456872683352033529421007878\backslash\\
29141860830768725102385452609882503551811073140339908096068\backslash\\
8125590506176016285837338837682469$
\end{itemize}

The primes $p_8$, $p_9$ and $p_{10}$ are far too large to show in this paper --- for instance $p_{10}$ has 5528 decimal digits. The primes $p_i$ for $i\leq8$ were verified using a deterministic primality test in Wolfram Mathematica 11 with the \texttt{ProvablePrimeQ} function in the PrimalityProving package, while $p_9$ and $p_{10}$ were certified prime by the Primo software \cite{Primo}. The certification of $p_{10}$ took $14$ hours and $23$ minutes on an Intel i7-4770 CPU and 4GB RAM. The prime certificates for $p_9$ and $p_{10}$ as well as the primes themselves can be found alongside this paper as auxiliary files.

The value of $B$ was calculated up to its first $5500$ decimal digits. The first $600$ are presented below:

\begin{center}
\begin{tabular}{rllll}
$1.2405547052$& $5201424067$& $4695153379$& $0034521235$& $3396725255$ \\
  $9232034386$& $1886622104$& $9111642316$& $9209174137$& $7064313608$ \\ 
  $3109555650$& $9480848158$& $9481662421$& $8378961303$& $7426392535$ \\
  $6658242301$& $8524802142$& $1960037621$& $1464734105$& $8229918628$ \\
  $4182439221$& $9437396337$& $9442594273$& $8936874985$& $9158491115$ \\
  $7886891108$& $4262398559$& $2731605607$& $5719554304$& $2915944781$ \\
  $6278755834$& $4774412491$& $8125993063$& $4590081972$& $8945860313$ \\ 
  $1303247244$& $0907981721$& $7119324606$& $1009855753$& $6063847008$ \\
  $6985820925$& $6038920740$& $0817313213$& $1691077511$& $3322609476$ \\ 
  $3239264899$& $5703729933$& $8452155290$& $5152647430$& $8960522935$ \\
  $3735771869$& $0936560934$& $8000430515$& $4856069064$& $6309177739$ \\
  $2832001365$& $6550953673$& $1549789328$& $9032942357$& $7708168137$
\end{tabular}
\end{center}

\bigskip
\hrule
\bigskip

\noindent 2010 {\it Mathematics Subject Classification}: Primary 11A41; Secondary 11Y60, 11Y11. \\
\noindent \emph{Keywords: } prime-representing function, Mills' constant, prime number sequence.

\bigskip
\hrule
\bigskip

\noindent (Concerned with sequences
\seqnum{A051021} and
\seqnum{A051254}.)

\bigskip
\hrule
\bigskip

\vspace*{+.1in}
\noindent
Received  June 8 2017;
revised versions received  September 20 2017; September 26 2017.
Published in {\it Journal of Integer Sequences}, October 29 2017.

\bigskip
\hrule
\bigskip

\noindent
Return to
\htmladdnormallink{Journal of Integer Sequences home page}{http://www.cs.uwaterloo.ca/journals/JIS/}.
\vskip .1in

\end{document}